\documentclass[11pt,a4paper]{article}

\usepackage{lmodern}
\usepackage[OT2,T1]{fontenc}
\usepackage[cpRIM]{inputenc}
\usepackage{url}
\usepackage{amsmath}
\usepackage{array}
\usepackage[pdftex]{graphicx}
\usepackage{amsfonts}
\usepackage{amssymb}
\usepackage{amsthm}
\usepackage{hyperref}
\usepackage{color}
\usepackage{a4wide}

\vspace{1cm}
\newtheorem{Theorem}{Theorem}[section]

\newtheorem{Lemma}[Theorem]{Lemma}

\newtheorem{Conjecture}[Theorem]{Conjecture}

\theoremstyle{definition}

\def \bR {\mathbb{R}}

\def \bZ {\mathbb{Z}}

\def \bS {\mathbb{S}}

\def \cR {\mathcal{R}}

\def \fA {\mathfrak{A}}

\def \fG {\mathfrak{G}}

\def \fE {\mathfrak{E}}

\def \fF {\mathfrak{F}}

\def \indic#1 {1\!\!1_{\!#1}}

\def \wt {\widetilde}
\def \Ba {\begin{aligned}}
\def \Ea {\end{aligned}}
\allowdisplaybreaks[4]

\begin{document}
\author{Kæstutis Èesnavièius}
\title{On a Generalization of the Flag Complex Conjecture of Charney and Davis}
\date{}
\maketitle
\begin{abstract}\noindent The Flag Complex Conjecture of Charney and Davis states that for a simplicial complex $S$ which triangulates a $(2n - 1)$-generalized homology sphere as a flag complex one has $(-1)^n \sum_{\sigma \in S} \left(\frac{-1}{2}\right)^{\dim\sigma + 1} \ge 0$, where the sum runs over all simplices $\sigma$ of $S$ (including the empty simplex). Interpreting the $1$-skeleta of $\sigma\in S$ as graphs of Coxeter groups, we present a stronger version of this conjecture, and prove the equivalence of the latter to the Flag Complex Conjecture.\end{abstract}

\section{Introduction}
Motivated by a longstanding Hopf-Thurston conjecture on the sign of the Euler characteristic of a closed, aspherical manifold of even dimension, in \cite{Charney} Charney and Davis conjectured a linear inequality that the number of faces of each dimension of a flag simplicial complex triangulating an odd-dimensional sphere should satisfy. Stanley \cite{Stanley}, Davis and Okun \cite{Okun}, Karu \cite{Karu} and Frohmader \cite{Frohmader} have settled partial cases of the conjecture, but the general case for spheres of dimension $n > 3$ remains open.\\

The Charney-Davis Conjecture is closely related to simplicial complexes arising from Coxeter groups, and is in fact a special case of the Orbifold Characteristic Conjecture for the Davis complex of a Coxeter group of type HM; see Chapter 16 of \cite{Davis} for details. In this paper we formulate a conjecture (Conjecture \ref{general} below), which reduces to the Charney-Davis Conjecture in the special case when all Coxeter groups involved are right-angled, and prove that this generalization is in fact as strong as the original conjecture.

\section{Coxeter Groups}
A Coxeter group $W$ is a group given by a presentation $W = \langle C\vert\cR\rangle$, with a set of generating involutions $C = \{c_1, c_2, \dotsc, c_n\}$, and a set of relations $\cR = \{ c_i^2 = 1\}\cup\{ (c_ic_j)^{m_{ij}} = 1, i \neq j, 2 \le m_{ij} \le \infty \}$ (where $m_{ij} = \infty$ designates the absence of a relation of the type $(c_ic_j)^k = 1$). Note that $m_{ij} = 2$ entails the commutativity of $c_i$ and $c_j$: $c_jc_i = c_i^2c_jc_ic_j^2 = c_i(c_ic_j)^2c_j = c_ic_j$; similarly, $m_{ij} = m_{ji}$. The Coxeter group $W$ is called \emph{right-angled} if all $m_{ij}$ are either $2$ or $\infty$.\\

To every Coxeter group $W$ one associates its \emph{Coxeter graph} $\Gamma$, which succinctly carries all the information needed to define a Coxeter group by means of generators and relations. For its construction, index the vertices of $\Gamma$ by the elements of $C$ and join two vertices $c_i$ and $c_j$ by an edge iff $m_{ij} \ge 3$; label the edge with $m_{ij}$ if the latter is $\ge 4$. Therefore, the absence of an edge between $c_i$ and $c_j$ stands for $m_{ij} = 2$ (i.e., $c_i$ and $c_j$ commute), whereas the absence of a label means $m_{ij} = 3$. Note that given $\Gamma$ one can unequivocally restore $W$. If $\Gamma$ has several connected components, then the corresponding $W$ decomposes as a direct product. Coxeter groups $W$ with connected graphs $\Gamma$ are called \emph{irreducible}; in deciding whether $\Gamma$ gives rise to a finite $W$ the classification of all finite Coxeter groups (see, e.g., the Chapter 2 of \cite{Humphreys}) comes in handy. We record the finite irreducible $W$ in terms of their Coxeter graphs in the Appendix.

\section{The Flag Complex Conjecture}
A topological space $X$ is a \emph{homology $n$-manifold} if it has the same local homology groups as $\bR^n$, i.e., $H_i(X, X - x) = H_i(\bR^n, \bR^n - \{0\}) = \begin{cases}\bZ\quad \mbox{if }i = n,\\ 0\quad \mbox{otherwise},\end{cases}$ for each $x \in X$.
If in addition $X$ has the same homology as the $n$-sphere $H_i(X) = H_i(\bS^n) = \begin{cases}\bZ\quad \mbox{if }i = 0\mbox{ or }i = n,\\ 0\quad \mbox{otherwise,}\end{cases}$ it is called a \emph{generalized homology $n$-sphere}, or a $\mathrm{GHS}^n$ for short.\\

If $\sigma$ is a simplex of a simplicial complex $S$, its \emph{closed star} $\mathrm{St}(\sigma)$ is the union of all simplices in $S$ that have $\sigma$ as a face. The \emph{link} $\mathrm{Lk}(\sigma)$ of $\sigma$ is the union of all faces $\tau$ of simplices in $\mathrm{St}(\sigma)$, such that $\sigma\cap\tau=\emptyset$. Note that $\mathrm{Lk(\sigma)}$ is itself a simplicial complex.\\

The following lemma gives a criterion for a simplicial complex $S$ to be a homology $n$-manifold in terms of the topology of the links of simplices of $S$. For the proof we refer to Lemma 10.4.6 in \cite{Davis}.
\begin{Lemma}\label{GHS}Suppose $S$ is an $n$-dimensional simplicial complex. The following are equivalent.
\begin{enumerate}
\item $S$ is a homology $n$-manifold.
\item For each simplex $\sigma$ of $S$, $\mathrm{Lk}(\sigma)$ is homeomorphic to a $\mathrm{GHS}^{n - \dim \sigma - 1}$.
\end{enumerate}\end{Lemma}

In the sequel we will be dealing with simplicial complexes $S$, the edges of which have assigned weights in the interval $2\le m\le \infty$ (the weight $\infty$ will signify the absence of an edge between corresponding vertices, whereas pairs of vertices with finite weight edges joining them will be said to be \emph{incident}). In this context the term \emph{flag complex} will stand for a simplicial complex $S$ in which every subset of pairwise incident edges spans a simplex. If $S$ is a finite flag complex, such that any simplex $\sigma \in S$ defines (by interpreting the $1$-skeleton of $\sigma$ as a Coxeter graph) a finite Coxeter group $W_{\sigma}$, we will set
\begin{equation}\label{omega}\tag{$\ddag$}\omega(S):=\sum\limits_{\sigma \in S} \frac{(-1)^{\dim{\sigma + 1}}}{\# W_{\sigma}}.\end{equation}
Note that if all the edges of $\sigma$ have weight $2$, the Coxeter group that $\sigma$ defines is $(\bZ/2\bZ)^{\dim\sigma + 1}$ and therefore $\# W_{\sigma} = 2^{\dim\sigma + 1}$. The following is the Flag Complex Conjecture of Charney and Davis formulated in \cite{Charney}.
\begin{Conjecture}[Charney-Davis]\label{CD}Suppose $S$ is a triangulation of a $(2n - 1)$-dimensional generalized homology sphere as a flag complex, where all the edges of $S$ have weight $2$ (the weights of absent edges are $\infty$). Then $(-1)^n\omega(S) \ge 0$.\end{Conjecture}
The conjecture admits a generalization, which we will be concerned with in this paper.
\begin{Conjecture}[Generalized Flag Complex Conjecture]\label{general}Suppose $S$ is a triangulation of a $(2n - 1)$-dimensional generalized homology sphere as a flag complex, and $W_{\sigma}$ is finite for every simplex $\sigma \in S$. Then $(-1)^n\omega(S) \ge 0$.\end{Conjecture}
We will prove that the two conjectures are equivalent, i.e., to prove the generalized version it suffices to prove the right-angled case.
\section{Modifying the Weights}
We say that a simplex $\sigma$ of $S$ is of type $G$, if $W_{\sigma}$ is $G$ ($G$ is one of the names $A_n$, $F_4$, $H_3$, etc. from the classification table in the Appendix). Note that in the case of a flag complex $S$ decreasing the weight of an edge does not change the set of simplices of $S$: from the classification table in the Appendix it is easy to see that $W_{\sigma}$ stays finite for $\sigma$ for which it was finite before this modification.\\

\begin{Lemma}\label{first}Suppose that the Conjecture \ref{general} holds for $n < k$, and $S$ is a triangulation of a generalized homology $(2k - 1)$-sphere as a flag complex such that $W_{\sigma}$ is finite whenever $\sigma \in S$. Suppose $e$ is an edge of $S$ of weight $6\le m < \infty$. Then in order to verify Conjecture \ref{general} for $S$, it suffices to verify it for $S^{\prime}$, obtained by replacing the weight of $e$ by $2 \le u < m$.\end{Lemma}
\begin{proof}From the classification in the Appendix it is clear that whenever an edge $e$ of $S$ of weight $6\le m < \infty$ belongs to a simplex $\sigma$ of $S$, the vertices of $e$ are joined to the other vertices of $\sigma$ with weight $2$ edges. If we replace $e$ with an edge of weight $u$, the change in $\omega(S)$ will be (see formula (\ref{omega}))
\[\frac{1}{2u}\omega(\mathrm{Lk}(e)) - \frac{1}{2m}\omega(\mathrm{Lk}(e)) = \frac{m - u}{2mu}\omega(\mathrm{Lk}(e)).\]
By Lemma \ref{GHS} $\mathrm{Lk}(e)$ is a generalized homology $(2k - 3)$-sphere, so by assumption Conjecture \ref{general} is true for $\mathrm{Lk}(e)$ (it is a flag complex for which $W_{\sigma}$ is finite for $\sigma\in\mathrm{Lk}(e)$). That is, the change in $\omega(S)$ is of the opposite sign than the value predicted by Conjecture \ref{general}. The claim follows.\end{proof}

\begin{Lemma}\label{weight5}In the setting of Lemma \ref{first}, suppose $S$ has no finite-weight edges of weight $\ge 6$, and $e$ is of weight $5$. Then, in order to verify Conjecture \ref{general} for $S$, it suffices to verify it for $S^{\prime}$ obtained by changing the weight of $e$ to $4$.\end{Lemma}
\begin{proof}Replacing $e$ with an edge of weight $4$ changes $\omega(S)$ by (summations run over all simplices $\sigma$ of $S$ of the indicated type which have $e$ as an edge) \[\begin{aligned}&\left(\frac{1}{8}\omega(\mathrm{Lk}(e)) + \sum_{\sigma:H_3}\left(\frac{1}{16} - \frac{1}{48}\right)\omega(\mathrm{Lk}(\sigma)) + \sum_{\sigma:H_4}\left(-\frac{1}{48} + \frac{1}{2}\left(\frac{1}{16} - \frac{1}{48}\right) + \frac{1}{384}\right)\omega(\mathrm{Lk}(\sigma))\right) - \\ &\left(\frac{1}{10}\omega(\mathrm{Lk}(e)) + \sum_{\sigma:H_3}\left(\frac{1}{20} - \frac{1}{120}\right)\omega(\mathrm{Lk}(\sigma)) + \sum_{\sigma:H_4}\left(-\frac{1}{60} + \frac{1}{2}\left(\frac{1}{20} - \frac{1}{120}\right) + \frac{1}{14400}\right)\omega(\mathrm{Lk}(\sigma))\right).\end{aligned}\]
Using an inclusion-exclusion type of argument the two summands calculate the contribution of all the simplices $\sigma$ of $S$ in which $e$ is involved in formula (\ref{omega}), when the weight of $e$ is declared $4$ and $5$, respectively. The coefficients in front of $\omega(\mathrm{Lk}(\sigma))$ account for all previous terms in which the contribution of simplices $\tau \supset \sigma$ have been calculated with incorrect weights, and corrects them for those $\tau$ in which the vertices of $\sigma$ are joined to the remaining vertices of $\tau$ by edges of weight $2$.\\

After rearranging the difference becomes
\[ \frac{1}{40}\omega(\mathrm{Lk}(e)) - \sum_{\sigma:H_4}\frac{47}{28800}\omega(\mathrm{Lk}(\sigma)). \]
Using Lemma \ref{GHS} and the assumptions we conclude that the change is of the opposite sign than the value of $\omega(S)$ predicted by Conjecture \ref{general}.\end{proof} 

\begin{Lemma}In the setting of Lemma \ref{first}, suppose $S$ has no finite-weight edges of weight $\ge 5$, and $e$ is of weight $4$. Then, in order to verify Conjecture \ref{general} for $S$, it suffices to verify it for $S^{\prime}$ obtained by changing the weight of $e$ to $4$.\end{Lemma}
\begin{proof}When $e$ is replaced by an edge of weight $3$, using an argument analogous to the one in the previous lemma, we find the contribution of all simplices $\sigma$ containing $e$ in formula (\ref{omega}) to be \[\sum_{\sigma: B_n}b_n\cdot\omega(\mathrm{Lk}(\sigma)) + \sum_{\sigma: F_4}f_4\cdot\omega(\mathrm{Lk}(\sigma)),\] where the coefficients $b_n$, $f_4$ satisfy recurrence relations (the argument is a generalized version of the one used in the proof of Lemma \ref{weight5}; $b_n$ corrects for the summands in which the contribution of $\sigma$ of the type $B_n$ was calculated incorrectly and adds the correct contribution; the logical meaning of summation variable $j$, looking at the sketch of the Coxeter graph $\Gamma$ for $B_n$ in the Appendix, is the number $n - j$ of nodes that are cut-off from the left to get the new graph, the previously counted contribution of which the summand adjusts)
\[\begin{aligned} b_n &= \sum_{2\le j < n}(-1)^{n - j + 1}\frac{b_j}{(n - j + 1)!} + \frac{(-1)^n}{(n + 1)!},\quad\quad n\ge 2,\\ f_4 &= -\frac{b_2}{4} + 2\cdot\frac{b_3}{2} + \frac{1}{120}.\end{aligned}\]
Similarly, when the weight of $e$ is $4$, the contribution of corresponding simplices is
\[\sum_{\sigma: B_n}\widetilde{b}_n\cdot\omega(\mathrm{Lk}(\sigma)) + \sum_{\sigma: F_4}\widetilde{f}_4\cdot\omega(\mathrm{Lk}(\sigma)),\]
with
\[\begin{aligned} \widetilde{b}_n &= \sum_{2\le j < n}(-1)^{n - j + 1}\frac{\widetilde{b}_j}{(n - j + 1)!} + \frac{(-1)^n}{2^nn!},\quad\quad n\ge 2,\\ \widetilde{f}_4 &= -\frac{\widetilde{b}_2}{4} + 2\cdot\frac{\widetilde{b}_3}{2} + \frac{1}{2^7\cdot 3^2}.\end{aligned}\]
The difference is
\[\sum_{\sigma: B_n}(b_n - \widetilde{b}_n)\cdot\omega(\mathrm{Lk}(\sigma)) + \sum_{\sigma: F_4}(f_4 - \widetilde{f}_4)\cdot\omega(\mathrm{Lk}(\sigma)).\]
By Lemma \ref{GHS} combined with the assumptions, $\omega(\mathrm{Lk}(\sigma))$ is of the same sign as the value of $\omega(S)$ predicted by Conjecture \ref{general} for $\sigma$ of types $B_{4l}$, $F_4$, and of the opposite sign for $\sigma$ of types $B_{4l + 2}$. To prove that the change is of the opposite sign than the predicted sign of $\omega(S)$, it suffices to prove that $(-1)^{l}(b_{2l} - \widetilde{b}_{2l}), f_4 - \widetilde{f}_4\le 0$ and $b_{2l + 1} - \widetilde{b}_{2l + 1} = 0$, $l\ge 1$. We set $\beta_n = (-1)^n(b_n - \widetilde{b}_n)$ and note the recurrence relation
\[ \beta_n = -\sum_{2\le j < n}\frac{\beta_j}{(n - j + 1)!} + \frac{1}{(n + 1)!} - \frac{1}{2^nn!},\quad\quad n\ge 2.\]
Considering the generating formal power series $\mathfrak{B}(x) = \beta_2x^2 + \beta_3x^3 + \dotsb$ and an auxiliary series $\mathfrak{F}(x) = \frac{x}{2!} + \frac{x^2}{3!} + \dotsb = \frac{e^x - x - 1}{x}$ we get the relation
\[\begin{aligned}\mathfrak{B}(x) &= -\mathfrak{B}(x)\mathfrak{F}(x) + \sum_{n = 2}^{\infty}\left(\frac{x^n}{(n + 1)!} - \frac{x^n}{2^n n!}\right) = -\mathfrak{B}(x)\mathfrak{F}(x) + \frac{1}{x}\left(e^x - 1 - x\right) - \left(e^{x/2} - 1\right),\\ \mathfrak{B}(x) &= \frac{x}{e^x - 1}\cdot \left(\frac{e^x - 1}{x} - e^{x/2}\right) = 1 - \frac{x}{e^{x/2} + 1} - \frac{x}{e^x - 1}.\end{aligned}\]
But $\frac{x}{e^x - 1} = \sum\limits_{n = 0}^{\infty} \frac{B_nx^n}{n!}$, where $B_n$ are the familiar Bernoulli numbers: $B_0 = 1$, $B_1 = -\frac{1}{2}$, $B_2 = \frac{1}{6}$, $B_3 = 0$, $B_4 = -\frac{1}{30}$, etc.; also, $\frac{2y}{e^y + 1} = \sum\limits_{n = 0}^{\infty} \frac{G_ny^n}{n!}$, where $G_n = 2(1 - 2^n)B_n$ are the Genocchi numbers (sequence A036968 in \cite{Sloane}). We conclude that $\beta_n = -\frac{G_n}{2^nn!} - \frac{B_n}{n!} = \frac{B_n}{n!}(1-\frac{1}{2^{n - 1}})$, and the conclusion follows from the well-known properties of Bernoulli numbers: $B_{2l + 1} = 0$ and $(-1)^{l}B_{2l}\le 0$ for $l\ge 1$.\\

Finally, $f_4 - \widetilde{f}_4 = -\frac{1}{4}\beta_2 + \beta_3 + \frac{43}{5760} = -\frac{17}{5760} \le 0$, as desired.\end{proof} 

To handle the case of weight $3$ edges we will need a couple of technical results on Bernoulli numbers; we carry them out in two subsequent lemmas.
\begin{Lemma}\label{ugly2}For $n \ge 2$, we have
\[\sum_{i = 1}^{n - 2} \frac{B_{n - i}}{i!(n - i)!} = \frac{1}{2(n - 1)!} - \frac{1}{n!}.\]
\end{Lemma}
\begin{proof}The left hand side is a coefficient at $x^n$ in the formal power series expansion of
\[(e^x - 1) \left(\frac{x}{e^x - 1} - 1 + \frac{x}{2}\right) = \frac{x}{2} + 1 + e^x\left(\frac{x}{2} - 1\right).\qedhere\]
\end{proof} 
\begin{Lemma}\label{ugly1}For $n\ge 3$, the Bernoulli numbers satisfy
\[\sum_{j = 2}^{n - 2}\frac{B_{n - j}}{2^{j + 1}(j + 2)!(n - j)!}+ \frac{B_{n - 1}}{24(n - 1)!} + \frac{B_n}{4n!} +\frac{B_{n + 1}}{(n + 1)!}   + \frac{B_{n + 2}}{(n + 2)!}\left(4 - \frac{1}{2^n}\right) = \frac{n + 1}{2^{n + 1}(n + 2)!}. \]\end{Lemma}
\begin{proof}The first three terms on the left hand side combined give
\[\sum_{j = 0}^{n}\frac{B_{n - j}}{2^{j + 1}(j + 2)!(n - j)!} + \frac{1}{2^{n + 1}(n + 1)!} -\frac{1}{2^{n + 1}(n + 2)!}, \]
so it remains to prove that
\[\sum_{j = 0}^{n}\frac{B_{n - j}}{2^{j + 1}(j + 2)!(n - j)!}+\frac{B_{n + 1}}{(n + 1)!}   + \frac{B_{n + 2}}{(n + 2)!}\left(4 - \frac{1}{2^n}\right) = 0.\]
The first term on the left is the coefficient at $x^n$ in the formal power series expansion of
\[\frac{2x}{e^x - 1} \cdot \frac{e^{x/2} - 1 - x/2}{x^2} = \frac{2}{x(e^{x/2} + 1)} - \frac{1}{e^x - 1},\]
which is
\[\frac{2G_{n + 2}}{2^{n + 2}(n + 2)!} - \frac{B_{n + 1}}{(n + 1)!} = \frac{(1 - 2^{n + 2})B_{n + 2}}{2^n(n + 2)!} - \frac{B_{n + 1}}{(n + 1)!} = -\left(4 - \frac{1}{2^n}\right)\frac{B_{n + 2}}{(n + 2)!} - \frac{B_{n + 1}}{(n + 1)!}.\qedhere \]
\end{proof} 
We are now ready for the weight $3$ case.
\begin{Lemma}In the setting of Lemma \ref{first}, suppose $S$ has no finite-weight edges of weight $\ge 4$, and $e$ is of weight $3$. Then, in order to verify Conjecture \ref{general} for $S$, it suffices to verify it for $S^{\prime}$ obtained by changing the weight of $e$ to $2$.\end{Lemma}
\begin{proof}We use a technique analogous to the previous lemma to calculate the contribution of all simplices $\sigma$ containing $e$ in formula (\ref{omega}) when the weight of $e$ is $2$ and $3$, respectively. We will need a slight modification of the argument presented there, however, because in the present case $e$ need not be the end-edge in the Coxeter graph of $W_{\sigma}$.\\

We say that $\sigma$ is of the type $A_n^t$ if the Coxeter graph $\Gamma$ of $W_{\sigma}$ is $A_n$, and $e$ is the $t^{\mathrm{th}}$ edge from the end in $\Gamma$ (note that this makes $\sigma$ be of the type $A_n^{n - t}$, too). Similarly, if $\Gamma$ is $D_n$ and $e$ is one of end-edges at the branched end of $\Gamma$, we say that $\sigma$ is of the type $D_n^{\prime}$; if $e$ is the $t^{\mathrm{th}}$, $t \ge 2$ edge from the branched end of $\Gamma$, we say that $\sigma$ is of the type $D_n^t$ (therefore, $\sigma$ can be of types $D_4^2$ and $D_4^{\prime}$ simultaneously). Types $E_6^{\prime}$, $E_7^{\prime}$, $E_8^{\prime}$ are defined analogously when $e$ is the vertical edge (in the sense of the diagrams presented in the Appendix) in graphs for $E_6, E_7, E_8$; types $E_n^t$, $6 \le n \le 8$ stand for those cases when $e$ is the $t^{\mathrm{th}}$ horizontal edge from the left (in the same sense) in the graph for $E_n$.\\

The contribution of all the summands in formula (\ref{omega}) corresponding to the simplices containing $e$, when its weight is $2$ is (the coefficients appearing in the sum will be discussed later)
\[\begin{aligned}\sum_{\sigma: A_n^t}a_{n, t}\cdot\omega(\mathrm{Lk}(\sigma)) &+ \sum_{\sigma: D_n^{\prime}}d_n^{\prime}\cdot\omega(\mathrm{Lk}(\sigma)) + \sum_{\sigma: D_n^t}d_{n, t}\cdot\omega(\mathrm{Lk}(\sigma))\\ &+ \sum_{\sigma: E_n^{\prime}}e_n^{\prime}\cdot\omega(\mathrm{Lk}(\sigma)) + \sum_{\sigma: E_n^t}e_{n, t}\cdot\omega(\mathrm{Lk}(\sigma)), \end{aligned}\]
where each $\sigma$ of type $A_n^t$ is counted either in $\sum\limits_{\sigma:A_n^t}$ or in $\sum\limits_{\sigma:A_n^{n - t}}$, and each $\sigma$ of type $D_4^{\prime}$ is counted either in $\sum\limits_{\sigma:D_4^2}$ or $\sum\limits_{\sigma:D_4^{\prime}}$, but not both (the recurrences for $a_{n, t}$ imply that $a_{n, t} = a_{n, n - t}$, similarly $d_4^{\prime} = d_{4, 2}$, so the choice does not matter; see below). When the weight of $e$ is $3$ the corresponding sum is
\[\begin{aligned}\sum_{\sigma: A_n^t}\wt{a}_{n, t}\cdot\omega(\mathrm{Lk}(\sigma)) &+ \sum_{\sigma: D_n^{\prime}}\wt{d}_n^{\prime}\cdot\omega(\mathrm{Lk}(\sigma)) + \sum_{\sigma: D_n^t}\wt{d}_{n, t}\cdot\omega(\mathrm{Lk}(\sigma))\\ &+ \sum_{\sigma: E_n^{\prime}}\wt{e}_n^{\prime}\cdot\omega(\mathrm{Lk}(\sigma)) + \sum_{\sigma: E_n^t}\wt{e}_{n, t}\cdot\omega(\mathrm{Lk}(\sigma)). \end{aligned}\]
To show that the change
\[\begin{aligned}\sum_{\sigma: A_n^t}(a_{n, t} - \widetilde{a}_{n, t})\cdot\omega(\mathrm{Lk}(\sigma)) &+ \sum_{\sigma: D_n^{\prime}}(d_n^{\prime} - \widetilde{d}_n^{\prime})\cdot\omega(\mathrm{Lk}(\sigma)) + \sum_{\sigma: D_n^t}(d_{n, t} - \widetilde{d}_{n, t})\cdot\omega(\mathrm{Lk}(\sigma)) \\ &+ \sum_{\sigma: E_n^{\prime}}(e_n^{\prime} - \widetilde{e}_n^{\prime})\cdot\omega(\mathrm{Lk}(\sigma)) + \sum_{\sigma: E_n^t}(e_{n, t} - \widetilde{e}_{n, t})\cdot\omega(\mathrm{Lk}(\sigma)) \end{aligned}\]
is of the sign opposite to the sign of $\omega(S)$ predicted by Conjecture \ref{general}, by Lemma \ref{GHS} and the assumptions it suffices to show that
\begin{align}(-1)^l(a_{2l, t} - \widetilde{a}_{2l, t}) &\le 0,\quad (a_{2l + 1, t} - \widetilde{a}_{2l + 1, t}) = 0\quad \text{for}\quad l \ge 1,\label{alpha} \\ (-1)^l(d_{2l}^{\prime} - \widetilde{d}_{2l}^{\prime}) &\le 0,\quad (d_{2l + 1}^{\prime}- \widetilde{d}_{2l + 1}^{\prime}) = 0\quad \text{for}\quad l \ge 2,\label{deltaprime} \\ (-1)^l(d_{2l, t} - \widetilde{d}_{2l, t}) &\le 0,\quad (d_{2l + 1, t} - \widetilde{d}_{2l + 1, t}) = 0\quad \text{for}\quad l \ge 2,\label{delta} \\ e_6^{\prime} - \widetilde{e}_6^{\prime} &\ge 0,\quad e_{6, t} - \widetilde{e}_{6, t} \ge 0,\label{esix} \\ e_7^{\prime} - \widetilde{e}_7^{\prime} &= 0,\quad e_{7, t} - \widetilde{e}_{7, t} = 0,\label{eseven} \\ e_8^{\prime} - \widetilde{e}_8^{\prime} &\le 0,\quad e_{8, t} - \widetilde{e}_{8, t} \le 0.\label{eeight}\end{align}\\
We turn our attention to the coefficients $a_{n, t}$ and $\widetilde{a}_{n, t}$. For $\sum\limits_{\sigma: A_n^t}a_{n, t}\cdot\omega(\mathrm{Lk}(\sigma))$ to calculate the contribution of all $\sigma$ of types $A_n^t$ correctly, $a_{n, t}$ must satisfy recurrence relations ($a_{n, t}$ corrects for the summands in which the contribution of $\sigma$ of the type $A_n^t$ was calculated incorrectly and adds the correct contribution; the logical meaning of summation variables $i$ and $j$, looking at the sketch of the Coxeter graph for $A_n$ in the Appendix and assuming $e$ is the $t^{\mathrm{th}}$ from the left, is the number of nodes that are cut-off from the left and the right, respectively, to get the new graph, the previously counted contribution of which the summand adjusts; the argument is a generalized version of the one used in the proof of Lemma \ref{weight5})
\[a_{n, t} = \underset{i + j \neq 0}{\sum_{i = 0}^{t - 1}\sum_{j = 0}^{n - t - 1}}\frac{(-1)^{i + j + 1}a_{n - i - j, t - i}}{(i + 1)!(j + 1)!} + \frac{(-1)^n}{(t + 1)!(n - t + 1)!}.\]
Similarly,
\[\widetilde{a}_{n, t} = \underset{i + j \neq 0}{\sum_{i = 0}^{t - 1}\sum_{j = 0}^{n - t - 1}}\frac{(-1)^{i + j + 1}\widetilde{a}_{n - i - j, t - i}}{(i + 1)!(j + 1)!} + \frac{(-1)^n}{(n + 1)!}.\]
Set $\alpha_{n, t} = (-1)^n(a_{n, t} - \widetilde{a}_{n, t})$. The recurrence relation
\[\begin{aligned}\alpha_{n, t} &= -\underset{i + j \neq 0}{\sum_{i = 0}^{t - 1}\sum_{j = 0}^{n - t - 1}}\frac{\alpha_{n - i - j, t - i}}{(i + 1)!(j + 1)!} + \frac{1}{(t + 1)!(n - t + 1)!} - \frac{1}{(n + 1)!}\\ &= -\sum_{s = 1}^{n - 2}\sum_{\substack{0\le i \le s \\ 1\le t - i \le n - s - 1}}\frac{\alpha_{n - s, t - i}}{(i + 1)!(s - i + 1)!} + \frac{1}{(t + 1)!(n - t + 1)!} - \frac{1}{(n + 1)!}\end{aligned}\]
yields a recurrence for the polynomials $P_n(y) = \alpha_{n, 1}y + \alpha_{n, 2}y^2 + \dotsb + \alpha_{n, n - 1}y^{n - 1}$. For $n\ge 2$,
\[ P_n(y) = -\frac{1}{y}\sum_{s = 1}^{n - 2}P_{n - s}(y)\frac{(1 + y)^{s + 2} - y^{s + 2} - 1}{(s + 2)!} + \sum_{t = 1}^{n - 1}\left(\frac{y^t}{(t + 1)!(n - t + 1)!} - \frac{y^t}{(n + 1)!}\right), \]
where the second sum equals
\[\frac{(1 + y)^{n + 2} - 1 - (n + 2)y - (n + 2)y^{n + 1} - y^{n + 2}}{y(n + 2)!} - \frac{1}{(n + 1)!}\frac{y^n - y}{y - 1},\]
which is
\[\frac{(1 + y)^{n + 2} - 1 - y^{n + 2}}{y(n + 2)!} - \frac{1}{(n + 1)!}\frac{y^{n+1} - 1}{y - 1}.\]
Letting $\mathfrak{G}(x) = \frac{x^3}{3!} + \frac{x^4}{4!} + \dotsb = e^x - 1 - x - \frac{x^2}{2}$ be an auxiliary formal power series, we conclude that the generating formal power series $\fA(x, y) = P_2(y)x^2 + P_3(y)x^3 + \dotsb$ satisfies the recurrence relation
\[\begin{split}\fA(x, y) = &-\frac{1}{x^2y}\fA(x, y)(\fG(x(y + 1))-\mathfrak{G}(xy)- \mathfrak{G}(x)) + \frac{\fG(x) - \fG(xy)}{x(y - 1)}\\ &+ \frac{1}{x^2y}\left(\fG(x(y + 1)) - \frac{(x(y + 1))^3}{6} - \fG(x) +  \frac{x^3}{6} - \fG(xy) + \frac{(xy)^3}{6}\right).\end{split}\]
But $\fG(x(y + 1))-\mathfrak{G}(xy)- \mathfrak{G}(x) = e^{x(y + 1)} - e^{xy} - e^{x} + 1 - x^2y$, so \[\begin{aligned}\frac{\fA(x, y)(e^{xy} - 1)(e^x - 1)}{x^2y} &= \frac{\fG(x)- \fG(xy)}{x(y - 1)} + \frac{(e^{xy} - 1)(e^x - 1)}{x^2y} - 1 - \frac{xy}{2} - \frac{x}{2}\\
&= \frac{e^x - e^{xy}}{x(y - 1)} + 1 + \frac{x(y + 1)}{2} +\frac{(e^{xy} - 1)(e^x - 1)}{x^2y}- 1 - \frac{xy}{2} - \frac{x}{2}.\end{aligned}\]
We conclude that
\[ \fA(x, y) = 1 + \frac{xy}{y - 1}\left(\frac{1}{e^{xy} - 1} - \frac{1}{e^x - 1}\right) = 1 + \sum_{n = 0}^{\infty} \frac{B_n(y^n - y)}{n!(y - 1)}x^n, \]
and $P_n(y) = \frac{B_n}{n!}(y + y^2 + \dotsb + y^{n - 1})$. Hence $\alpha_{n, t} = \frac{B_n}{n!}$ is independent of $t$, and has the same sign as the $n^{\mathrm{th}}$ Bernoulli number $B_n$; odd Bernoulli numbers vanish, whereas $(-1)^lB_{2l}\le 0$, (\ref{alpha}) follows.\\

Consider the numbers $d_n^{\prime}$ and $\wt{d}_n^{\prime}$ and let $\delta_n^{\prime} = (-1)^n(d_n^{\prime}-\widetilde{d}_n^{\prime})$. The recurrence relations (the first and the second sums correspond to subgraphs containing both forked edges and only one forked edge, respectively)
\[\begin{aligned}d_n^{\prime} &= \left(\sum_{i = 1}^{n - 4}(-1)^{i + 1}\frac{d_{n - i}^{\prime}}{(i + 1)!} + \frac{(-1)^{(n - 3) + 1}a_{3, 1}}{(n - 2)!}\right) + \sum_{i = 0}^{n - 3}(-1)^i\frac{a_{n - i - 1, 1}}{2\cdot (i + 1)!} + \frac{(-1)^n}{2\cdot n!}, \\
\widetilde{d}_n^{\prime} &= \left(\sum_{i = 1}^{n - 4}(-1)^{i + 1}\frac{\widetilde{d}_{n - i}^{\prime}}{(i + 1)!} + \frac{(-1)^{(n - 3) + 1}a_{3, 1}}{(n - 2)!}\right) + \sum_{i = 0}^{n - 3}(-1)^i\frac{\wt{a}_{n - i - 1, 1}}{2\cdot (i + 1)!} + \frac{(-1)^n}{2^{n - 1}\cdot n!},\end{aligned}\]
together with $\alpha_{n, 1} = \frac{B_n}{n!}$ give
\[\Ba\delta_n^{\prime} &= -\sum_{i = 1}^{n - 4}\frac{\delta_{n - i}^{\prime}}{(i + 1)!} - \frac{1}{2}\sum_{i = 0}^{n - 3}\frac{B_{n - i - 1}}{(n - i - 1)! (i + 1)!} + \frac{1}{2\cdot n!} - \frac{1}{2^{n - 1}\cdot n!} \\
&= -\sum_{i = 1}^{n - 4}\frac{\delta_{n - i}^{\prime}}{(i + 1)!} - \frac{1}{4(n - 1)!} + \frac{1}{n!} - \frac{1}{2^{n - 1}n!},\Ea\]
where we have made use of Lemma \ref{ugly2} in the last equality. Recall the auxiliary series $\fF(x) = \frac{x}{2!} + \frac{x^2}{3!} + \dotsb = \frac{e^x - x - 1}{x}$; the generating formal power series $\fE(x) = \delta_4^{\prime}x^4 + \delta_5^{\prime}x^5 + \dotsb$ satisfies
\[\begin{aligned}
\fE(x) &= -\fE(x)\fF(x) - \frac{x}{4}\left(e^x - 1 - x - \frac{x^2}{2}\right) + \left(e^x - \sum_{n = 0}^3\frac{x^n}{n!}\right) - 2\left(e^{x/2} - \sum_{n = 0}^3 \frac{x^n}{2^nn!}\right), \\
\fE(x) &= \frac{x}{e^x - 1}\left( -\frac{x}{4}e^x + \frac{x}{4} + \frac{x^2}{4} + \frac{x^3}{8} + e^x - 2e^{x/2} + 1 - \frac{x^2}{4} - \frac{x^3}{8}\right) \\ &= \frac{x}{e^x - 1}\left(\frac{x(1 - e^x)}{4} + (e^{x/2} - 1)^2\right) = -\frac{x^2}{4} + x - \frac{2x}{e^{x/2} + 1}. \end{aligned}\]
We encounter the Genocchi numbers $\frac{2y}{e^y + 1} = \sum\limits_{n = 0}^{\infty}\frac{G_ny^n}{n!}$ again, and conclude that $\delta_n^{\prime} = -\frac{2G_n}{2^nn!} = -\frac{2(1 - 2^n)B_n}{2^{n - 1}n!} = \frac{B_n}{n!}\left(4 - \frac{1}{2^{n -2}}\right)$ has the same sign as the $n^{\mathrm{th}}$ Bernoulli number $B_n$; this proves (\ref{deltaprime}).\\

In fact, we will prove that $(-1)^n(d_{n, t} - \widetilde{d}_{n, t}) = \frac{B_n}{n!}\left(4 - \frac{1}{2^{n -2}}\right)$, too, and this will settle (\ref{delta}). We begin with the case $t = 2$. The recurrence relations are (the sums correspond to subgraphs containing both, just one, and neither of the forked edges, respectively)
\[\begin{aligned}d_{n, 2} &= \sum_{i = 1}^{n - 4}(-1)^{i + 1}\frac{d_{n - i, 2}}{(i + 1)!} + 2\sum_{i = 0}^{n - 4}(-1)^i\frac{a_{n - i - 1, 2}}{2(i + 1)!} + \sum_{i = 0}^{n - 4}(-1)^{i + 1}\frac{a_{n - i - 2, 1}}{4(i + 1)!}+\frac{(-1)^n}{24(n - 2)!},\\ \widetilde{d}_{n, 2} &= \sum_{i = 1}^{n - 4}(-1)^{i + 1}\frac{\widetilde{d}_{n - i, 2}}{(i + 1)!} + 2\sum_{i = 0}^{n - 4}(-1)^i\frac{\widetilde{a}_{n - i - 1, 2}}{2(i + 1)!} + \sum_{i = 0}^{n - 4}(-1)^{i + 1}\frac{\widetilde{a}_{n - i - 2, 1}}{4(i + 1)!}+\frac{(-1)^n}{2^{n - 1}n!}.\end{aligned}\]
Letting $\delta_{n, t} = (-1)^n(d_{n, t} - \widetilde{d}_{n, t})$ we see that
\[\Ba\delta_{n, 2} &= -\sum_{i = 1}^{n - 4}\frac{\delta_{n - i, 2}}{(i + 1)!} - \sum_{i = 0}^{n - 4}\frac{\alpha_{n - i - 1, 2}}{(i + 1)!} - \sum_{i = 0}^{n - 4} \frac{\alpha_{n - i - 2, 1}}{4(i + 1)!} + \frac{1}{24(n - 2)!} - \frac{1}{2^{n - 1}n!} \\
&= -\sum_{i = 1}^{n - 4}\frac{\delta_{n - i, 2}}{(i + 1)!}-\sum_{i = 1}^{n - 2}\frac{B_{n - i}}{(n - i)!i!} + \frac{B_2}{2(n - 2)!} - \frac{1}{4}\sum_{i = 1}^{n - 3}\frac{B_{n - i - 1}}{(n - i - 1)!i!} + \frac{1}{24(n - 2)!} - \frac{1}{2^{n - 1}n!} \\
&= -\sum_{i = 1}^{n - 4}\frac{\delta_{n - i, 2}}{(i + 1)!}-\frac{1}{2(n - 1)!}+ \frac{1}{n!} - \frac{1}{8(n - 2)!} + \frac{1}{4(n -1 )!} + \frac{1}{8(n - 2)!} - \frac{1}{2^{n - 1}n!}\\
&= -\sum_{i = 1}^{n - 4}\frac{\delta_{n - i, 2}}{(i + 1)!} - \frac{1}{4(n - 1)!} + \frac{1}{n!} - \frac{1}{2^{n - 1}n!},\Ea\]
where we have used Lemma \ref{ugly2}. Since the recurrence relation for the numbers $\delta_{n, 2}$ is identical to that for $\delta_n^{\prime}$, we conclude that $\delta_{n, 2} = \delta_n^{\prime} = \frac{B_n}{n!}\left(4 - \frac{1}{2^{n - 2}}\right)$, as desired.\\

The general case $n - 2\ge t\ge 3$ is handled similarly, though the relations are slightly different (the sums correspond to cases when the subgraph has both forked edges, one forked edge, none forked edges but contains the branch point, none forked edges and ends on the right side one edge off the branch point, none forked edges and ends at the right side $j + 1$ edges off the branch point, respectively):
\[\begin{aligned}d_{n, t} = &\sum_{i = 1}^{n - t - 2}(-1)^{i + 1}\frac{d_{n - i, t}}{(i + 1)!} + 2\sum_{i = 0}^{n - t - 2}(-1)^i\frac{a_{n - i - 1, t}}{2(i + 1)!} + \sum_{i = 0}^{n - t - 2}(-1)^{i + 1}\frac{a_{n - i - 2, t-1}}{4(i + 1)!}+ \\ &\sum_{i = 0}^{n - t - 2}(-1)^i\frac{a_{n - i - 3, t - 2}}{24(i + 1)!} + \sum_{i = 0}^{n - t - 2}\sum_{j = 1}^{t - 3}(-1)^{i + j} \frac{a_{n - i - j - 3, t - j - 2}}{2^{j + 2}(j + 3)!(i + 1)!}+ \frac{(-1)^n}{2^t(t + 1)!(n - t)!},\\ \widetilde{d}_{n, t} = &\sum_{i = 1}^{n - t - 2}(-1)^{i + 1}\frac{\widetilde{d}_{n - i, t}}{(i + 1)!} + 2\sum_{i = 0}^{n - t - 2}(-1)^i\frac{\widetilde{a}_{n - i - 1, t}}{2(i + 1)!} + \sum_{i = 0}^{n - t - 2}(-1)^{i + 1}\frac{\widetilde{a}_{n - i - 2, t-1}}{4(i + 1)!}+ \\ &\sum_{i = 0}^{n - t - 2}(-1)^i\frac{\widetilde{a}_{n - i - 3, t - 2}}{24(i + 1)!} + \sum_{i = 0}^{n - t - 2}\sum_{j = 1}^{t - 3}(-1)^{i + j} \frac{\widetilde{a}_{n - i - j - 3, t - j - 2}}{2^{j + 2}(j + 3)!(i + 1)!}+ \frac{(-1)^n}{2^{n - 1}n!}.\end{aligned}\]
Therefore,
\[\begin{aligned}\delta_{n, t} = &-\sum_{i = 1}^{n - t - 2}\frac{\delta_{n - i, t}}{(i + 1)!} - \sum_{i = 0}^{n - t - 2}\frac{\alpha_{n - i - 1, t}}{(i + 1)!} - \sum_{i = 0}^{n - t - 2}\frac{\alpha_{n - i - 2, t-1}}{4(i + 1)!} - \sum_{i = 0}^{n - t - 2}\frac{\alpha_{n - i - 3, t - 2}}{24(i + 1)!}\\ &- \sum_{i = 0}^{n - t - 2}\sum_{j = 1}^{t - 3} \frac{\alpha_{n - i - j - 3, t - j - 2}}{2^{j + 2}(j + 3)!(i + 1)!}+ \frac{1}{2^t(t + 1)!(n - t)!} - \frac{1}{2^{n - 1}n!}.\end{aligned}\]
We show that $\delta_{t + 2, t} = \frac{B_{t + 2}}{(t + 2)!}\left(4 - \frac{1}{2^t}\right)$, and $\delta_{n, t + 1} = \delta_{n, t}$ for $n - 2\ge t + 1$, $t \ge 3$. These two claims combined will give $\delta_{n, t} = \delta_{n, n - 2} = \frac{B_{n}}{n!}\left(4 - \frac{1}{2^{n - 2}}\right)$ and (\ref{delta}) will be settled. For the first one,
\[\begin{aligned}\delta_{t + 2, t} &= -\alpha_{t + 1, t} - \frac{\alpha_{t, t - 1}}{4} - \frac{\alpha_{t - 1, t - 2}}{24} - \sum_{j = 1}^{t - 3}\frac{\alpha_{t - j - 1, t - j - 2}}{2^{j + 2}(j + 3)!} + \frac{1}{2^{t + 1}(t + 1)!} - \frac{1}{2^{t + 1}(t + 2)!}\\ &= -\frac{B_{t + 1}}{(t + 1)!} - \frac{B_t}{4t!} - \frac{B_{t - 1}}{24(t - 1)!} - \sum_{j = 1}^{t - 3}\frac{B_{t - j - 1}}{2^{j + 2}(j + 3)!(t - j - 1)!} + \frac{t + 1}{2^{t + 1}(t + 2)!},\end{aligned}\]
which equals $\frac{B_{t + 2}}{(t + 2)!}\left(4 - \frac{1}{2^{t}}\right)$ by Lemma \ref{ugly1}. For the second one,
\[\begin{aligned} \delta_{n, t + 1} - \delta_{n, t} = &\frac{\delta_{t + 2, t}}{(n - t - 1)!} + \frac{\alpha_{t + 1, t}}{(n - t - 1)!} + \frac{\alpha_{t, t - 1}}{4(n - t - 1)!} + \frac{\alpha_{t - 1, t - 2}}{24(n - t - 1)!} +\\ &\sum_{j = 1}^{t - 3}\frac{\alpha_{t - j - 1, t - j - 2}}{2^{j + 2}(j + 3)!(n - t - 1)!} - \sum_{i = 0}^{n - t - 3} \frac{\alpha_{n - t - i - 1, 1}}{2^t(t + 1)!(i + 1)!}  + \frac{n - t - 2(t + 2)}{2^{t + 1}(t + 2)!(n - t)!},\end{aligned}\]
and employing Lemma \ref{ugly1} again, $2^t(t + 1)!(\delta_{n, t + 1} - \delta_{n, t})$ equals
\[\begin{aligned}\frac{t + 1}{2(t + 2)(n - t - 1)!} - \sum_{i = 0}^{n - t - 3} \frac{\alpha_{n - t - i - 1, 1}}{(i + 1)!} + \frac{n - t - 2(t + 2)}{2(t + 2)(n - t)!},\end{aligned}\]
which is zero by Lemma \ref{ugly2}. We have therefore established (\ref{delta}).\\

Finally, we analyze the sporadic cases. In the calculations we will need the numerical values of the first few Bernoulli numbers; we record them here: $B_0 = 1$, $B_1 = -\frac{1}{2}$, $B_2 = \frac{1}{6}$, $B_3 = 0$, $B_4 = -\frac{1}{30}$, $B_5 = 0$, $B_6 = \frac{1}{42}$, etc. \\

For type $E_6^{\prime}$ the relations are
\[\begin{aligned}e_6^{\prime} &= -\frac{1}{36}a_{2, 1} + \frac{2}{12}a_{3, 1} - \frac{2}{6}a_{4, 1} - \frac{1}{4}d_{4, 2} + \frac{1}{2}d_5^{\prime} + \frac{1}{2\cdot 6!}, \\ \widetilde{e}_6^{\prime} &= -\frac{1}{36}\widetilde{a}_{2, 1} + \frac{2}{12}\widetilde{a}_{3, 1} - \frac{2}{6}\widetilde{a}_{4, 1} - \frac{1}{4}\widetilde{d}_{4, 2} + \frac{1}{2}\widetilde{d}_5^{\prime} + \frac{1}{2^7\cdot 3^4 \cdot 5},\end{aligned}\]
and
\[ e_6^{\prime} - \widetilde{e}_6^{\prime} = -\frac{1}{36}\frac{B_2}{2} - \frac{1}{3}\frac{B_4}{24} - \frac{1}{4}\frac{B_4}{24}\left(4 - \frac{1}{4}\right) + \frac{7}{10368} = \frac{13}{103680} \ge 0. \]
Similarly,
\[\begin{aligned}e_{6, 1} &= -\frac{1}{120}a_{2, 1} + \frac{1}{12}a_{3, 1} - \frac{1}{4}a_{4, 1} - \frac{1}{6}a_{4, 1} + \frac{1}{2}a_{5, 1} + \frac{1}{2}d_{5, 3} + \frac{1}{2\cdot 2^4\cdot 5!}, \\ \widetilde{e}_{6, 1} &= -\frac{1}{120}\widetilde{a}_{2, 1} + \frac{1}{12}\widetilde{a}_{3, 1} - \frac{1}{4}\widetilde{a}_{4, 1} - \frac{1}{6}\widetilde{a}_{4, 1} + \frac{1}{2}\widetilde{a}_{5, 1} + \frac{1}{2}\widetilde{d}_{5, 3} + \frac{1}{2^7\cdot 3^4 \cdot 5},\end{aligned}\]
\[e_{6, 1} - \widetilde{e}_{6, 1} = -\frac{1}{120}\frac{B_2}{2} - \frac{1}{4}\frac{B_4}{24} - \frac{1}{6}\frac{B_4}{24} + \frac{5}{20736} = \frac{13}{103680} \ge 0.\]
\[\begin{aligned}e_{6, 2} &= -\frac{1}{24}a_{2, 1} + \frac{2}{2\cdot 6}a_{3, 1} + \frac{1}{8}a_{3, 1} - \frac{1}{6}a_{4, 1} - \frac{2}{4}a_{4, 1} - \frac{1}{4}d_{4, 2} + \frac{1}{2}a_{5, 2} + \frac{1}{2}d_{5, 2} + \frac{1}{2}d_5^{\prime} + \frac{1}{720}, \\ \widetilde{e}_{6, 2} &= -\frac{1}{24}\widetilde{a}_{2, 1} + \frac{2}{2\cdot 6}\widetilde{a}_{3, 1} + \frac{1}{8}\widetilde{a}_{3, 1} - \frac{1}{6}\widetilde{a}_{4, 1} - \frac{2}{4}\widetilde{a}_{4, 1} - \frac{1}{4}\widetilde{d}_{4, 2} + \frac{1}{2}\widetilde{a}_{5, 2} + \frac{1}{2}\widetilde{d}_{5, 2} + \frac{1}{2}\widetilde{d}_5^{\prime} + \frac{1}{2^7\cdot 3^4 \cdot 5},\end{aligned}\]
\[e_{6, 2} - \widetilde{e}_{6, 2} = -\frac{1}{24}\frac{B_2}{2}-\frac{1}{6}\frac{B_4}{24} - \frac{1}{2}\frac{B_4}{24} - \frac{1}{4}\frac{B_4}{24}\left(4 - \frac{1}{4}\right) + \frac{71}{51840} = \frac{13}{103680} \ge 0.\]
By symmetry $e_{6, 3} - \widetilde{e}_{6, 3} = e_{6, 2}- \widetilde{e}_{6, 2}$, $e_{6, 4} - \widetilde{e}_{6, 4} = e_{6, 1}- \widetilde{e}_{6, 1}$, so this establishes (\ref{esix}).\\

We proceed with types $E_7^{\prime}$ and $E_7^t$; as the method of calculation is clear by now, for the sake of brevity we consider the differences directly, skipping summands corresponding to odd number of vertices which vanish by the above.
\begin{align*} e_7^{\prime} - \widetilde{e}_7^{\prime} &= \frac{1}{6\cdot 24} \alpha_{2, 1} + \frac{1}{24}\alpha_{4, 1} + \frac{1}{12}\delta_4^{\prime} + \frac{1}{12}\alpha_{4, 1} + \frac{1}{2}(e_6^{\prime} - \widetilde{e}_6^{\prime}) + \frac{1}{2}\delta_6^{\prime} - \frac{1}{2\cdot 7!} + \frac{1}{2^{10}\cdot 3^4\cdot 5\cdot 7} \\ &= \frac{1}{6\cdot 24} \frac{B_2}{2} + \frac{1}{8}\frac{B_4}{24} + \frac{1}{12}\frac{B_4}{24}\left(4 - \frac{1}{4}\right) + \frac{1}{2}\frac{13}{103680} + \frac{1}{2}\frac{B_6}{6!}\left(4 - \frac{1}{2^4}\right) - \frac{41}{414720} = 0,\\
e_{7, 1} - \widetilde{e}_{7, 1} &= \frac{1}{6!}\alpha_{2, 1} + \frac{1}{24}\alpha_{4, 1} + \frac{1}{12}\alpha_{4, 1} + \frac{1}{2}\alpha_{6, 1} + \frac{1}{2}(e_{6, 1} - \widetilde{e}_{6, 1}) - \frac{1}{2\cdot 2^5 \cdot 6!} + \frac{1}{2^{10}\cdot 3^4\cdot 5\cdot 7} \\ &= \frac{1}{6!}\frac{B_2}{2} + \frac{1}{8}\frac{B_4}{24} + \frac{1}{2}\frac{B_6}{6!} + \frac{1}{2}\frac{13}{103680} - \frac{31}{1451520} = 0,\\
e_{7, 2} - \widetilde{e}_{7, 2} &= \frac{1}{96}\alpha_{2, 1} + \frac{1}{24}\alpha_{4, 2} + \frac{1}{12}\alpha_{4, 2} + \frac{1}{8}\alpha_{4, 1} + \frac{1}{12}\delta_{4, 2} + \frac{1}{2}\alpha_{6, 2} + \frac{1}{2}(e_{6, 2} - \widetilde{e}_{6, 2}) + \frac{1}{2}\delta_6^{\prime} \\ &-\frac{1}{6\cdot 6!} + \frac{1}{2^{10}\cdot 3^4\cdot 5\cdot 7} = \frac{1}{96}\frac{B_2}{2} + \frac{1}{4}\frac{B_4}{24} + \frac{1}{12}\frac{B_4}{4!}\left(4 - \frac{1}{4}\right) + \frac{1}{2}\frac{B_6}{6!} + \frac{1}{2}\frac{13}{103680} \\ &+ \frac{1}{2}\frac{B_6}{6!}\left(4 - \frac{1}{2^4}\right) - \frac{671}{2903040} = 0,\\
e_{7, 3} - \widetilde{e}_{7, 3} &= \frac{1}{72}\alpha_{2, 1} + \frac{2}{12}\alpha_{4, 1} + \frac{1}{8}\alpha_{4, 2} + \frac{1}{12}\alpha_{4, 2} + \frac{1}{12}\delta_{4, 2} + \frac{1}{2}\alpha_{6, 3} + \frac{1}{2}(e_{6, 2} - \widetilde{e}_{6, 2}) + \frac{1}{2}\delta_{6, 2} \\ &-\frac{1}{24\cdot 120} + \frac{1}{2^{10}\cdot 3^4\cdot 5\cdot 7} = \frac{1}{72}\frac{B_2}{2} + \frac{3}{8}\frac{B_4}{24} + \frac{1}{12}\frac{B_4}{4!}\left(4 - \frac{1}{4}\right) + \frac{1}{2}\frac{B_6}{6!}\\ &+ \frac{1}{2}\frac{13}{103680} + \frac{1}{2}\frac{B_6}{6!}\left(4 - \frac{1}{2^4}\right) - \frac{1007}{2903040} = 0, \\
e_{7, 4} - \widetilde{e}_{7, 4} &= \frac{1}{240}\alpha_{2, 1} + \frac{1}{12}\alpha_{4, 2} + \frac{1}{8}\alpha_{4, 1} + \frac{1}{12}\alpha_{4, 1} + \frac{1}{2}\alpha_{6, 2} + \frac{1}{2}(e_{6, 1} - \widetilde{e}_{6, 1}) + \frac{1}{2}\delta_{6, 3}\\ &-\frac{1}{6\cdot 2^4 \cdot 5!} + \frac{1}{2^{10}\cdot 3^4\cdot 5\cdot 7} = \frac{1}{240}\frac{B_2}{2} + \frac{7}{24}\frac{B_4}{24} + \frac{1}{2}\frac{B_6}{6!} + \frac{1}{2}\frac{13}{103680}\\ &+ \frac{1}{2}\frac{B_6}{6!}\left(4 - \frac{1}{2^4}\right) - \frac{251}{2903040} = 0,\\
e_{7, 5} - \widetilde{e}_{7, 5} &= \frac{1}{2^4\cdot 5!}\alpha_{2, 1} + \frac{1}{12}\alpha_{4, 1} + \frac{1}{2}\alpha_{6, 1} + \frac{1}{2}\delta_{6, 4} - \frac{1}{2\cdot 2^7 \cdot 3^4 \cdot 5} + \frac{1}{2^{10}\cdot 3^4\cdot 5\cdot 7}\\ &= \frac{1}{1920}\frac{B_2}{2} + \frac{1}{12}\frac{B_4}{24} + \frac{1}{2}\frac{B_6}{6!} + \frac{1}{2}\frac{B_6}{6!}\left(4 - \frac{1}{2^4}\right) - \frac{1}{107520} = 0,\end{align*}
and this settles (\ref{eseven}).\\

Finally, we turn our attention to $E_8^{\prime}$ and $E_8^t$.
\begin{align*} e_8^{\prime} - \widetilde{e}_8^{\prime} &= -\frac{1}{6\cdot 120} \alpha_{2, 1} - \frac{1}{120}\alpha_{4, 1} - \frac{1}{48}\delta_4^{\prime} - \frac{1}{36}\alpha_{4, 1} - \frac{1}{6}(e_6^{\prime} - \widetilde{e}_6^{\prime}) - \frac{1}{4}\delta_6^{\prime} - \frac{1}{6}\alpha_{6, 1} + \frac{1}{2\cdot 8!}\\ &- \frac{1}{2^{14}\cdot 3^5\cdot 5^2\cdot 7} = -\frac{1}{720} \frac{B_2}{2} - \frac{13}{360}\frac{B_4}{24} - \frac{1}{48}\frac{B_4}{24}\left(4 - \frac{1}{4}\right) - \frac{1}{6}\frac{13}{103680}\\ &- \frac{1}{4}\frac{B_6}{6!}\left(4 - \frac{1}{2^4}\right) - \frac{1}{6}\frac{B_6}{6!} + \frac{8639}{696729600} = -\frac{2537}{696729600}\le 0,\\
e_{8, 1} - \widetilde{e}_{8, 1} &= -\frac{1}{7!} \alpha_{2, 1} - \frac{1}{120}\alpha_{4, 1} - \frac{1}{48}\alpha_{4, 1} - \frac{1}{6}(e_{6, 1} - \widetilde{e}_{6, 1}) - \frac{1}{4}\alpha_{6, 1} + \frac{1}{2\cdot 2^6\cdot 7!} - \frac{1}{2^{14}\cdot 3^5\cdot 5^2\cdot 7}\\ &= -\frac{1}{5040} \frac{B_2}{2} - \frac{7}{240}\frac{B_4}{24} - \frac{1}{6}\frac{13}{103680} - \frac{1}{4}\frac{B_6}{6!} + \frac{1079}{696729600} = -\frac{2537}{696729600}\le 0,\\
e_{8, 2} - \widetilde{e}_{8, 2} &= -\frac{1}{4\cdot 120}\alpha_{2, 1} - \frac{1}{120}\alpha_{4, 2} - \frac{1}{48}\alpha_{4, 2} - \frac{1}{24}\alpha_{4, 1} - \frac{1}{48}\delta_{4, 2} - \frac{1}{4}\alpha_{6, 2} - \frac{1}{4}\alpha_{6, 1} - \frac{1}{6}(e_{6, 2} - \widetilde{e}_{6, 2})\\ &- \frac{1}{4}\delta_6^{\prime} +\frac{1}{6\cdot 7!} - \frac{1}{2^{14}\cdot 3^5\cdot 5^2\cdot 7} = -\frac{1}{480}\frac{B_2}{2} - \frac{17}{240}\frac{B_4}{24} - \frac{1}{48}\frac{B_4}{4!}\left(4 - \frac{1}{4}\right) - \frac{1}{2}\frac{B_6}{6!}\\ &-\frac{1}{6}\frac{13}{103680} - \frac{1}{4}\frac{B_6}{6!}\left(4 - \frac{1}{2^4}\right) + \frac{23039}{696729600} = -\frac{2537}{696729600} \le 0,\\
e_{8, 3} - \widetilde{e}_{8, 3} &= -\frac{1}{2\cdot 6 \cdot 24}\alpha_{2, 1} - \frac{1}{48}\alpha_{4, 1} - \frac{1}{48}\delta_{4, 2} - \frac{1}{36}\alpha_{4, 2} - \frac{1}{24}\alpha_{4, 2} - \frac{1}{24}\alpha_{4, 1} - \frac{1}{4}\alpha_{6, 3} - \frac{1}{6}(e_{6, 2} - \widetilde{e}_{6, 2})\\ &- \frac{1}{4}\delta_{6, 2} - \frac{1}{4}\alpha_{6, 2} - \frac{1}{6}\alpha_{6, 2} +\frac{1}{120\cdot 120} - \frac{1}{2^{14}\cdot 3^5\cdot 5^2\cdot 7} = -\frac{1}{288}\frac{B_2}{2} - \frac{19}{144}\frac{B_4}{24} - \frac{2}{3}\frac{B_6}{6!}\\ &- \frac{1}{48}\frac{B_4}{4!}\left(4 - \frac{1}{4}\right)  - \frac{1}{4}\frac{B_6}{6!}\left(4 - \frac{1}{2^4}\right) - \frac{1}{6}\frac{13}{103680} + \frac{48383}{696729600} = -\frac{2537}{696729600} \le 0, \\
e_{8, 4} - \widetilde{e}_{8, 4} &= -\frac{1}{6\cdot 120}\alpha_{2, 1} - \frac{1}{120}\alpha_{4, 1} - \frac{1}{24}\alpha_{4, 2} - \frac{1}{24}\alpha_{4, 1} - \frac{1}{36}\alpha_{4, 1} - \frac{1}{4}\alpha_{6, 3} - \frac{1}{6}\alpha_{6, 3} - \frac{1}{4}\delta_{6, 3} - \frac{1}{4}\alpha_{6, 2} \\ &- \frac{1}{6}(e_{6, 1} - \widetilde{e}_{6, 1}) +\frac{1}{24\cdot 2^4 \cdot 5!} - \frac{1}{2^{14}\cdot 3^5\cdot 5^2\cdot 7} = -\frac{1}{720}\frac{B_2}{2} - \frac{43}{360}\frac{B_4}{24} - \frac{2}{3}\frac{B_6}{6!} - \frac{1}{6}\frac{13}{103680}\\ &- \frac{1}{4}\frac{B_6}{6!}\left(4 - \frac{1}{2^4}\right) + \frac{15119}{696729600} = -\frac{2537}{696729600} \le 0,\\
e_{8, 5} - \widetilde{e}_{8, 5} &= -\frac{1}{2\cdot 2^4\cdot 5!}\alpha_{2, 1} - \frac{1}{120}\alpha_{4, 2} - \frac{1}{24}\alpha_{4, 1} - \frac{1}{4}\alpha_{6, 2} - \frac{1}{6}\alpha_{6, 2} - \frac{1}{4}\delta_{6, 4} - \frac{1}{4}\alpha_{6, 1} + \frac{1}{6\cdot 2^7 \cdot 3^4 \cdot 5}\\ &- \frac{1}{2^{14}\cdot 3^5\cdot 5^2\cdot 7} = -\frac{1}{3840}\frac{B_2}{2} - \frac{1}{20}\frac{B_4}{24} - \frac{2}{3}\frac{B_6}{6!} - \frac{1}{4}\frac{B_6}{6!}\left(4 - \frac{1}{2^4}\right) + \frac{2239}{696729600}\\ &= -\frac{2537}{696729600} \le 0,\\
e_{8, 6} - \widetilde{e}_{8,6} &= -\frac{1}{2^7\cdot 3^4\cdot 5}\alpha_{2, 1} - \frac{1}{120}\alpha_{4, 1} - \frac{1}{4}\alpha_{6, 1} - \frac{1}{6}\alpha_{6, 1} + \frac{1}{2\cdot 2^{10} \cdot 3^4 \cdot 5\cdot 7} - \frac{1}{2^{14}\cdot 3^5\cdot 5^2\cdot 7}\\ &= -\frac{1}{51840}\frac{B_2}{2} - \frac{1}{120}\frac{B_4}{24} - \frac{5}{12}\frac{B_6}{6!} + \frac{17}{99532800}= -\frac{2537}{696729600} \le 0,\end{align*}
and this gives (\ref{eeight}).\\

We have proved (\ref{alpha}), (\ref{deltaprime}), (\ref{delta}), (\ref{esix}), (\ref{eseven}), and (\ref{eeight}), and hence the claim.\end{proof} 

\section{The Equivalence of the Two Conjectures}
We are now in the position to prove the equivalence of Conjectures \ref{CD} and \ref{general}.
\begin{Theorem}\label{main}To prove the Generalized Flag Complex Conjecture, it suffices to prove the Charney-Davis Flag Complex Conjecture.\end{Theorem}
\begin{proof}Indeed, suppose the Charney-Davis Conjecture \ref{CD} held true. Then we could prove the Generalized Flag Complex Conjecture by induction. The base case $k = 0$ of the empty simplicial complex $S$ is clear. If $k > 1$ we would be in the setting of the lemmas of the preceding section. We could without loss of generality decrease the weights of the edges step by step, every time reducing the weight of an edge of the highest finite weight, until we would be left with $S^{\prime}$ in which all the finite weight edges would be of weight $2$. The claim of Conjecture \ref{CD} applied to $S^{\prime}$ would imply the claim of Conjecture \ref{general} to $S$.\end{proof}
\section*{Acknowledgements}
The work presented here originates from a guided research project for my Bachelor Thesis at Jacobs University Bremen. I would like to thank my supervisor Dr. Pavel Tumarkin for helpful discussions, as well as many comments and suggestions.
\newpage
\section*{Appendix}

\begin{table}[here]\label{classification}
 \centering
\begin{tabular}{|c|c|c|} \hline
 Name  & Coxeter graph & Cardinality \\ \hline \hline
 $A_n\ \ (n\ge 1)$& \scalebox{0.6}{\includegraphics{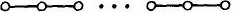}} &  $(n + 1)!$  \\ \hline
$ B_n\ \ (n \ge 2) $ & \scalebox{0.5}{\includegraphics{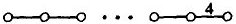}} & $2^nn!$ \\ \hline
$ D_n\ \ (n \ge 4) $ & \scalebox{0.23}{\includegraphics{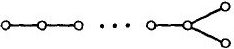}} & $2^{n - 1}n!$ \\ \hline
$ E_6 $ & \scalebox{0.23}{\includegraphics{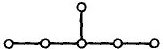}} & $2^7\cdot 3^4\cdot5$ \\ \hline
$ E_7 $ & \scalebox{0.23}{\includegraphics{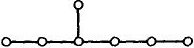}} & $2^{10}\cdot 3^4\cdot 5\cdot 7$ \\ \hline
$ E_8 $ & \scalebox{0.23}{\includegraphics{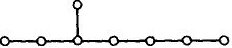}} & $2^{14}\cdot 3^5\cdot 5^2\cdot 7$ \\ \hline
$ F_4 $ & \scalebox{0.5}{\includegraphics{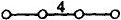}} & $2^7\cdot 3^2$ \\ \hline
$ H_3 $ & \scalebox{0.5}{\includegraphics{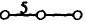}} & $120$ \\ \hline
$ H_4 $ & \scalebox{0.5}{\includegraphics{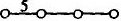}} & $14 400$ \\ \hline
$ I_2\ \ (m \ge 5) $ & \scalebox{0.5}{\includegraphics{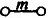}} & $2m$ \\ \hline
\end{tabular}
 \caption{All finite irreducible Coxeter groups}
\end{table}

\bibliographystyle{amsplain}	
\bibliography{myrefs}		

\medskip\noindent {\footnotesize Massachusetts Institute of Technology,\\ Department of Mathematics, Room 2-089, \\ 77 Massachusetts Avenue, \\ Cambridge, MA 02139, \\ United States of America \\}
{\footnotesize E-mail address: {\tt kestutis@mit.edu }}
\end{document}